\RequirePackage[l2tabu, orthodox]{nag}
\documentclass[11pt,oneside,reqno]{article}
\usepackage{tikz}
\usetikzlibrary{decorations.fractals,math}
\usepackage[small]{titlesec}

\usepackage%[a4paper]
{geometry}

\usepackage[final]{microtype}
\makeatletter
\def\MT@register@subst@font{\MT@exp@one@n\MT@in@clist\font@name\MT@font@list
   \ifMT@inlist@\else\xdef\MT@font@list{\MT@font@list\font@name,}\fi}
\makeatother 

\usepackage{mathrsfs}

\usepackage{lmodern}

\usepackage{amsmath, amssymb, amsfonts, amsthm, mathtools}
\allowdisplaybreaks

\usepackage{fixmath}
\mathtoolsset{centercolon}

\usepackage[latin1]{inputenc}
\usepackage[T1]{fontenc}

\usepackage[draft=false,pdftex]{hyperref}
\hypersetup{
    bookmarks=true,         % show bookmarks bar?
    unicode=false,          % non-Latin characters in Acrobat’s bookmarks
    pdftoolbar=true,        % show Acrobat’s toolbar?
    pdfmenubar=true,        % show Acrobat’s menu?
    pdffitwindow=false,     % window fit to page when opened
    pdfstartview={FitH},    % fits the width of the page to the window
    pdftitle={Outer linear measure of connected sets via Steiner trees},    % title
    pdfauthor={Konrad J. Swanepoel},     % author
    pdfnewwindow=false,      % links in new window
    colorlinks=true,       % false: boxed links; true: colored links
    linkcolor=black,          % color of internal links
    citecolor=black,        % color of links to bibliography
    filecolor=black,      % color of file links
    urlcolor=black           % color of external links
}

\newcommand{\setbuilder}[2]{\left\{#1\;\middle|\;#2\right\}}
\newcommand{\netbuilder}[2]{\left(#1\;\middle|\;#2\right)}
\newcommand{\set}[1]{\left\{#1\right\}}
\newcommand{\net}[1]{\left(#1\right)}
\newcommand{\epsi}{\varepsilon}
\newcommand{\fhi}{\varphi}
\newcommand{\card}[1]{\lvert#1\rvert}
\newcommand{\closure}[1]{\overline{#1}}
\DeclareMathOperator{\diam}{diam}
\DeclareMathOperator{\li}{Li}
\newcommand{\MST}{\operatorname{mst}}
\newcommand{\SMT}{\operatorname{smt}}
\newcommand{\numbersystem}[1]{\mathbb{#1}}
\newcommand{\bN}{\numbersystem{N}}
\newcommand{\bR}{\numbersystem{R}}
\newcommand{\CB}{\mathscr{B}}
\newcommand{\CC}{\mathscr{C}}
\newcommand{\CE}{\mathscr{E}}
\newcommand{\CG}{\mathscr{G}}
\newcommand{\CH}{\mathscr{H}}
\newcommand{\CP}{\mathscr{P}}

\newcommand{\CT}{\mathscr{T}}
\newcommand{\CV}{\mathscr{V}}
\newcommand{\CU}{\mathscr{U}}
\newcommand{\CX}{\mathscr{X}}

\theoremstyle{plain}
\newtheorem{theorem}{Theorem}

\newtheorem*{golabtheorem}{Go\l\k{a}b's Theorem}
\newtheorem{gtheorem}{Theorem}

\newtheorem*{choquettheorem}{Choquet's Theorem}

\newtheorem{corollary}[theorem]{Corollary}
\newtheorem{proposition}[theorem]{Proposition}
\newtheorem{lemma}[theorem]{Lemma}

\title{Outer linear measure of connected sets via Steiner trees}
\author{Konrad J.\ Swanepoel\thanks{Department of Mathematics, London School of Economics and Political Science, Houghton Street, London WC2A 2AE, United Kingdom. Email: 
    \href{mailto:k.swanepoel@lse.ac.uk}{k.swanepoel@lse.ac.uk}}}
\date{}
\begin{document}
\maketitle

\begin{abstract}
We resurrect an old definition of the linear measure of a metric continuum in terms of Steiner trees, independently due to Menger (1930) and Choquet (1938).
We generalise it to any metric space and provide a proof of a little-known theorem of Choquet that it coincides with the outer linear measure for any connected metric space.
As corollaries we obtain simple proofs of Go\l\k{a}b's theorem (1928) on the lower semicontinuity of linear measure of continua and a theorem of Bogn\'ar (1989) on the linear measure of the closure of a set.
We do not use any measure theory apart from the definition of outer linear measure.
\end{abstract}
\section{Introduction}
A well-known chestnut presents a sequence of arcs $\gamma_n\colon[0,1]\to\bR^2$ between fixed points $a$ and $b$ at distance $2$, each consisting of $2^{n-1}$ semicircular arcs of radius $2^{n-1}$, converging (uniformly) to the straight-line arc $\gamma_0\colon[0,1]\to\bR^2$, $\gamma_0(t)=(1-t)a+tb$ (Fig.~\ref{fig:circle}).
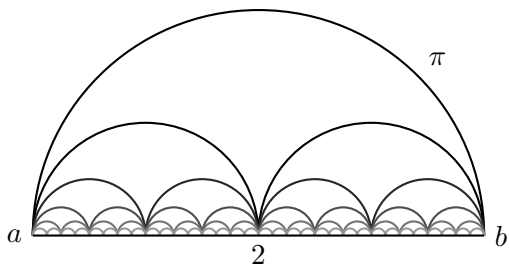
\begin{figure}[t]
\centering
\begin{tikzpicture}[thick, scale=3]
%    \fill[gray] (0,0) circle (0.5pt);
%    \draw[dotted, gray] (0,0) -- node[midway, below, black] {$1$} (1,0);
    \draw (0,0) node [below] {$2$};
%    \fill[gray] (1,0) circle (0.5pt);
    \draw (1,0) arc (0:180:1cm);
    \draw (45:1cm) node[above right] {$\pi$};
    \draw[black] (1,0) \foreach \x in {1,2} {arc (0:180:0.5cm)};
\draw[black!85] (1,0) \foreach \x in {1,...,4} {arc (0:180:0.25cm)};
\draw[black!70] (1,0) \foreach \x in {1,...,8} {arc (0:180:0.125cm)};
\draw[black!55] (1,0) \foreach \x in {1,...,16} {arc (0:180:0.0625cm)};
\draw[black!40] (1,0) \foreach \x in {1,...,32} {arc (0:180:0.03125cm)};
%\draw[black!25] (1,0) \foreach \x in {1,...,64} {arc (0:180:0.015625cm)};
%\draw[black!10] (1,0) \foreach \x in {1,...,128} {arc (0:180:0.0078125cm)};
\draw (-1,0) node[left] {$a$} -- (1,0) node[right] {$b$};
\end{tikzpicture}
\caption{A sequence of arcs, each of length $\pi$, converging to a segment of length $2$.}\label{fig:circle}
\end{figure}
That the length of each $\gamma_n$ is $\pi$ and of the limit curve is $2$ shows that arc length is not a continuous functional.
Instead, for any metric space $(X,d)$, arc length is a lower semicontinuous functional on the space $C([0,1],X)$ of continuous functions with the supremum norm.
This is most easily seen by noting that arc length can be defined as the supremum of the polygonal approximation functionals
\begin{equation}\label{eq:arclength}
\gamma\mapsto\sum_{i=1}^{n-1} d(\gamma(a_i),\gamma(a_{i+1})),
\end{equation}
where $0\leq a_1<\dots<a_n\leq 1$, each one of which is continuous, and applying the result that the supremum of a collection of lower semicontinuous functions is again lower semicontinuous.
Thus, the length of the limit arc can never be larger than the limit of the lengths of the arcs in the sequence, but can also be strictly smaller.

This observation also demonstrates that it is not completely straightforward that the Koch curve (Figure~\ref{koch0}) has infinite length, despite being the limit of curves $\kappa_n$ of length $(4/3)^n$, tending to infinity.
\begin{figure}
\centering
\begin{tikzpicture}[thin,decoration=Koch snowflake, scale=5]
    \draw[black!20] decorate{ (-1,0) -- (1,0) };
    \draw[black!30] decorate{ decorate{ (-1,0) -- (1,0) }};
    \draw[black!40] decorate{ decorate{ decorate{ (-1,0) -- (1,0) }}};
    \draw[thin,black!60] decorate{ decorate{ decorate{ decorate{ (-1,0) -- (1,0) }}}};
    \draw[thin,black!80] decorate{ decorate{ decorate{ decorate{ decorate{ (-1,0) -- (1,0) }}}}};
    \draw[thin,black] decorate{ decorate{ decorate{ decorate{ decorate{ decorate{ (-1,0) -- (1,0) }}}}}};
\end{tikzpicture}
\caption{The Koch curve}\label{koch0}
\end{figure}
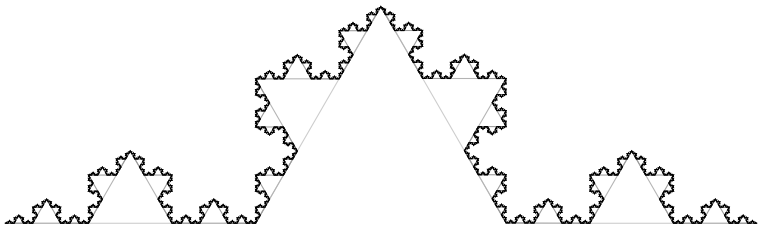
This point is usually glossed over in discussions of the Koch curve, including the original paper of Helge von Koch \cite{Koch1904}.
In Figure~\ref{koch2} we obtain $\delta_n$ by shrinking $\kappa_n$ by a linear factor of $(4/5)^{n-1}$ and adding segments at the two endpoints of the shrunken curve so that the endpoints are again at the same distance as before.
\begin{figure}
\centering
\begin{tikzpicture}[thin,decoration=Koch snowflake, scale=2.8]
    \draw decorate{ (-1,0) -- (1,0) } node[right=2mm] {$\delta_1$};

    \tikzmath{
        real \a, \b, \c;
        \a = 0.8;
        \b = \a;
        \c =-\b;
    }
    \draw[yshift=-0.4cm] (-1,0) -- (\c,0) decorate{ decorate{ -- (\b,0) }} -- (1,0) node[right=2mm] {$\delta_2$};

    \tikzmath{
        \b = \b * \a;
        \c =-\b;
    }
    \begin{scope}[yshift=-0.72cm]
        \draw (-1,0) -- (\c,0);
        \draw decorate{ decorate{ decorate{ (\c,0) -- (\b,0) }}};
        \draw (\b,0) -- (1,0) node[right=2mm] {$\delta_3$};
    \end{scope}

    \tikzmath{
        \b = \b * \a;
        \c =-\b;
    }
    \draw[yshift=-1cm] (-1,0) -- (\c,0) decorate{ decorate{ decorate{ decorate{ -- (\b,0) }}}} -- (1,0) node[right=2mm] {$\delta_4$};

    \tikzmath{
        \b = \b * \a;
        \c =-\b;
    }
    \begin{scope}[yshift=-1.25cm]
        \draw (-1,0) -- (\c,0);
        \draw[thin] decorate{ decorate{ decorate{ decorate{ decorate{ (\c,0) -- (\b,0) }}}}};
        \draw (\b,0) -- (1,0) node[right=2mm] {$\delta_5$};
    \end{scope}

    \tikzmath{
        \b = \b * \a;
        \c =-\b;
    }
    \begin{scope}[yshift=-1.5cm]
        \draw (-1,0) -- (\c,0);
        \draw[thin] decorate{ decorate{ decorate{ decorate{ decorate{ (\c,0) -- (\b,0) }}}}};
        \draw (\b,0) -- (1,0) node[right=2mm] {$\delta_6$};
    \end{scope}

    \tikzmath{
        \b = \b * \a;
        \c =-\b;
    }
    \begin{scope}[yshift=-1.7cm]
        \draw (-1,0) -- (\c,0);
        \draw[thin] decorate{ decorate{ decorate{ decorate{ decorate{ (\c,0) -- (\b,0) }}}}};
        \draw (\b,0) -- (1,0) node[right=2mm] {$\delta_7$};
    \end{scope}

    \tikzmath{
        \b = \b * \a;
        \c =-\b;
    }
    \begin{scope}[yshift=-1.9cm]
        \draw (-1,0) -- (\c,0);
        \draw[thin] decorate{ decorate{ decorate{ decorate{ decorate{ (\c,0) -- (\b,0) }}}}};
        \draw (\b,0) -- (1,0) node[right=2mm] {$\delta_8$};
    \end{scope}

    %\draw[yshift=-2cm] (-1,0) -- (1,0);
\end{tikzpicture}
\caption{Sequence of arcs of length $>(16/15)^n$ converging to a segment}\label{koch2}
\end{figure}
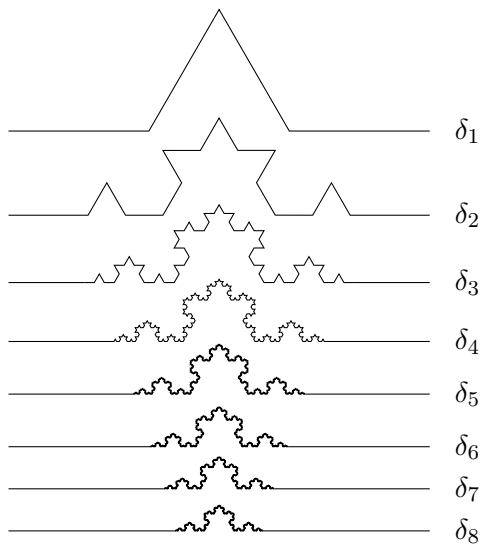
Then the length of $\delta_n$ is bounded from below by $(16/15)^{n}$, but the limit curve will be the segment $\gamma_0$.
In order to show that the Koch curve has infinite length, one has to observe that the vertices of each polygonal arc $\kappa_n$ occurs in all subsequent $\kappa_m$, $m>n$, hence also in the limit curve, and therefore by the elementary definition of arc length as the supremum of \eqref{eq:arclength}, the length of the Koch curve is bounded from below by the length of each $\kappa_n$.

The lower semicontinuity of arc length was generalised in the early $20$th century from arcs to \emph{continua}, that is, connected compact subsets of a metric space.
The first to publish a proof of the lower semicontinuity of the length of continua seems to be Stanis\l{}aw Go\l\k{a}b \cite{Golab1928}, and this theorem is nowadays known as Go\l\k{a}b's Theorem.
Recently, it has been used and generalised in existence proofs in the Calculus of Variations \cite{Dalmaso1992, DalmasoToader2002, PaoliniStepanov2013}.
See Section~\ref{history} for further historical remarks.

The main purpose of this paper is to give a simple, self-contained proof of Go\l\k{a}b's Theorem by using an almost forgotten definition of the length of continua due to Karl Menger \cite{Menger1931} and Gustave Choquet \cite{Choquet1938}.
This proof uses no measure theory apart from the definition of outer linear measure, and no results on the structure of continua.

\section{Five formulations of lower semicontinuity}
Let $(X,d)$ be a metric space.
Denote the open ball with centre $x\in X$ and radius $r>0$ by $B(x,r)=\setbuilder{y\in X}{d(x,y)<r}$ and
the power set of $X$ by
$\CP(X) = \setbuilder{A}{A\subseteq X}$.
For convenience we define $\sup\emptyset=0$ and $\inf\emptyset=\infty$.
Define the \emph{diameter} of $A\subseteq X$ as\[\diam(A)=\sup\setbuilder{d(a_1,a_2)}{a_1,a_2\in A}.\]
This gives an extended real-valued function $\diam\colon\CP(X)\to[0,\infty]$.

The \emph{outer linear measure} $L^*\colon\CP(X)\to[0,\infty]$ (also called $1$-dimensional outer Hausdorff measure) is defined as follows.
A countable family $\setbuilder{U_i}{i\in\bN}$ of subsets of $X$ is called a \emph{$\delta$-cover} of $A$ if $\diam(U_i)\leq\delta$ for all $i\in\bN$ and $A\subseteq\bigcup_{i\in\bN}U_i$.
For any $\delta>0$, let \[L^*_\delta(A)=\inf\setbuilder{\sum_{i\in\bN} \diam(U_i)}{\text{$\{U_i:i\in\bN\}$ is a $\delta$-cover of $A$}},\]
and let $L^*(A)=\sup\setbuilder{L^*_\delta(A)}{\delta>0} = \lim_{\delta\to 0}L^*_\delta(A)$.
Note that $L^*$ is \emph{monotone}: if $A\subseteq B$ then $L^*(A)\leq L^*(B)$.

Roughly, Go\l\k{a}b's theorem \cite{Golab1928} asserts that $L^*$ is lower semicontinuous on the collection of continua of $X$ with a suitable topology.
Instead of starting off with a discussion of these topologies, we give five concrete formulations of Go\l\k{a}b's Theorem, four of which can be found in the literature.
After giving a self-contained proof of the strongest of these, we discuss the topologies behind the different formulations.

We need the following definitions.
For $a\in X$ and $A,B\subseteq X$, define the \emph{distance from $a$ to $B$} as \[d(a,B)=\inf\setbuilder{d(a,b)}{b\in B},\]
the \emph{excess} (or \emph{Hausdorff hemimetric}) \emph{of $A$ over $B$} as \[e(A,B)=\sup\setbuilder{d(a,B)}{a\in A},\]
and the \emph{Hausdorff distance between $A$ and $B$} as \[D(A,B)=\max\set{e(A,B),e(B,A)}.\]
These again define extended real-valued functions $d\colon X\times\CP(X)\to[0,\infty]$, $e\colon \CP(X)^2\to[0,\infty]$ and $D\colon\CP(X)^2\to[0,\infty]$.
Note that, according to our definitions, $d(a,\emptyset)=e(A,\emptyset)=D(A,\emptyset)=\infty$ and $e(\emptyset,B)=0$.
A (metric) \emph{continuum} is a connected, compact metric space.

The following theorem was proved by Stanis\l{}aw Go\l\k{a}b in 1928 \cite{Golab1928}.
\begin{golabtheorem}
Let $A$ be a continuum and $\netbuilder{A_n}{n\in\bN}$ a sequence of continua in the metric space $X$ such that $\lim_{n\to\infty} D(A,A_n)=0$.
Then \[L^*(A)\leq\liminf\limits_{n\to\infty}L^*(A_n).\]
\end{golabtheorem}
There are various ways to strengthen this theorem.
It is not necessary to assume that $A$ is a continuum, as this follows from the convergence of $A_n$ to $A$ with respect to Hausdorff distance \cite[Theorem~3.18]{Falconer1986}.
Although Go\l\k{a}b's Theorem is usually formulated in terms of Hausdorff distance, Go\l\k{a}b already remarked that the theorem still holds if we only assume that $\lim_{n\to\infty} e(A,A_n)=0$.
However, this weaker assumption does not imply that $A$ is a continuum, as $\lim_{n\to\infty} e(A',A_n)=0$ holds for any $A'\subseteq A$.
This does not matter, as the compactness of $A$ or the $A_n$ is not essential for the conclusion to hold.
We thus state the following version of Go\l\k{a}b's Theorem.
\begin{gtheorem}\label{golab}
Let $\netbuilder{A_n}{n\in\bN}$ be a sequence of connected subsets of the metric space $X$ and $A\subseteq X$ such that $\lim_{n\to\infty} e(A,A_n)=0$.
Then \[L^*(A)\leq\liminf\limits_{n\to\infty}L^*(A_n).\]
\end{gtheorem}
Independently of Go\l\k{a}b, in 1936 Orrin Frink \cite{Frink1936} proved the lower semicontinuity of $L^*$ on continua in the following form.
Although he proved only the first statement, the second statement follows in a very similar way.
Again we formulate it for connected sets.
\begin{gtheorem}\label{frink}
Let $A$ be a subset of the metric space $X$.
\begin{enumerate}
\item If $L^*(A)<\infty$ then for all $\epsi>0$ there exists $\delta>0$ such that for any connected $B\subseteq X$ with $e(A,B)<\delta$, $L^*(A)<L^*(B)+\epsi$.
\item If $L^*(A)=\infty$, then for all $M>0$ there exists $\delta>0$ such that for any connected $B\subseteq X$ with $e(A,B)<\delta$, $L^*(B)>M$.
\end{enumerate}
\end{gtheorem}
Theorem~\ref{golab} follows easily from Theorem~\ref{frink}.
Indeed, if $L^*(A)=\infty$, then Theorem~\ref{frink} gives that for any $M>0$ there exists $n_0\in\bN$ such that for all $n\geq n_0$, $L^*(A_n)>M$, which implies $\liminf_{n\to\infty} L^*(A_n)=\infty$, and if $L^*(A)<\infty$, then Theorem~\ref{frink} gives that for any $\epsi>0$ there exists $n_0\in\bN$ such that for all $n\geq n_0$, $L^*(A)<L^*(A_n)+\epsi$, which implies $L^*(A)\leq\liminf_{n\to\infty} L^*(A_n)+\epsi$.

A third form of Go\l\k{a}b's Theorem may be found in a 1989 paper of M\'aty\'as Bogn\'ar \cite{Bognar1989}.
Although Bogn\'ar only formulated it for sequences, the generalisation to nets presents no problems.
Let $(\Lambda,\preccurlyeq)$ be a directed set and $\netbuilder{A_\lambda}{\lambda\in\Lambda}$ a net of subsets of $X$.
We say that $x\in X$ is a \emph{limit point} of $\net{A_\lambda}$ if for each $\epsi>0$ there exists $\lambda_0\in\Lambda$ such that for all $\lambda\succ \lambda_0$, $A_\lambda\cap B(x,\epsi)\neq\emptyset$.
Then the (Painlev\'e--Kuratowski) \emph{lower limit} of $\net{A_\lambda}$, denoted by $\li A_\lambda$, is defined to be the set  of all limit points of $\net{A_\lambda}$.
(The topology defined by this convergence is not necessarily first countable, hence the use of nets.)
\begin{gtheorem}\label{bognar}
Let $\net{A_\lambda}$ be a net of connected subsets of a metric space $X$.
Then \[L^*(\li A_\lambda)\leq\liminf_{\lambda}L^*(A_\lambda).\]
\end{gtheorem}
It is again easy to see that Theorem~\ref{bognar} implies Theorem~\ref{golab}.
Indeed, for any $\epsi>0$ there exists $n_0\in\bN$ such that $e(A,A_n)<\epsi$ for all $n\geq n_0$.
Thus for any $x\in A$ we have $d(x,A_n)<\epsi$, which implies that $x\in\li_n A_n$.
Then Theorem~\ref{bognar} and the monotonicity of $L^*$ gives that $L^*(A)\leq L^*(\li A_n)\leq\liminf_{n\to\infty} L^*(A_n)$.

A fourth form of Go\l\k{a}b's Theorem appears in a 1992 paper of David Fremlin \cite[Lemma~5C]{Fremlin1992}.
We say that a set $A$ \emph{hits} a collection $\CU$ of sets if $A\cap U\neq\emptyset$ for all $U\in\CU$.
\begin{gtheorem}\label{fremlin}
Let $A$ be a subset of the metric space $X$ and let $\alpha\geq 0$.
Suppose that for any finite collection $\CU$ of open sets in $X$ which is hit by $A$, there exists a connected $C\subset X$ that hits $\CU$, such that $L^*(C)\leq\alpha$.
Then $L^*(A)\leq\alpha$.
\end{gtheorem}
This theorem easily implies Theorem~\ref{bognar}.

We state a final version very close to Theorem~\ref{fremlin}, but in a form parallel to Theorem~\ref{frink}.
\begin{gtheorem}\label{lsc}
Let $A$ be a subset of a metric space $X$.
\begin{enumerate}
\item If $L^*(A)<\infty$ then for all $\epsi>0$ there exists a finite collection $\CU$ of open sets in $X$ hit by $A$ such that for any connected $B\subseteq X$ that hits $\CU$, $L^*(A)<L^*(B)+\epsi$.
\item If $L^*(A)=\infty$, then for all $M>0$ there exists a finite collection $\CU$ of open sets in $X$ hit by $A$ such that for any connected $B\subseteq X$ that hits $\CU$, $L^*(B)>M$.
\end{enumerate}
\end{gtheorem}
This theorem easily implies both Theorems~\ref{frink} and \ref{fremlin}.

In the next section we give a self-contained proof of Theorem~\ref{lsc}.
This is done in two steps.
First, we introduce the Menger--Choquet length of an arbitrary subset of a metric space and present a very simple proof of its semicontinuity formulated in a manner analogous to Theorem~\ref{lsc} (Proposition~\ref{MClsc}).
Then we state our main result (Theorem~\ref{maintheorem}), that the outer linear measure of a connected subset of a metric space $X$ equals its Menger--Choquet length.
This result generalises an old theorem of Choquet \cite{Choquet1938}.
Its proof is in Section~\ref{proof}.
Up to here everything is self-contained (using some elementary notions from graph theory).
In particular, everything is done without specifying exactly what we mean by lower semicontinuity.
In Section~\ref{topology} we discuss the two topological senses in which the above versions of Go\l\k{a}b's Theorem assert the lower semicontinuity of $L^*$ on connected sets or on continua.
Section~\ref{history} contains a discussion of the history and applications of Go\l\k{a}b's Theorem.

\section{Menger--Choquet length and Steiner trees}\label{tree}
It is clear that the outer linear measure of a set $A$, as defined in the previous section, is independent of the metric space of which $A$ is a subset.
The same will not be true for what we will call the Menger--Choquet length, introduced next.

As usual, we consider a \emph{graph} on a finite non-empty set of \emph{vertices} $V$ to be a pair $G=(V,E)$ where $E\subseteq\binom{V}{2}$, the set of \emph{edges} of $G$, is a set of unordered pairs of vertices.
The \emph{degree} of a vertex $x$ of $G$ is 
$\deg_G(x) = \card{\setbuilder{e\in E}{x\in e}}$.
A \emph{path} with \emph{endpoints} $a$ and $b$ is a subgraph of $G$ with distinct vertices $v_1,\dots,v_n$, $n\geq 1$, where $a=v_1$, $b=v_n$, and edges $\set{v_i,v_{i+1}}\in E$ for $i=1,\dots,n-1$.
A graph $G$ is \emph{connected} if for any $a,b\in V$, $G$ contains a path with endpoints $a$ and $b$, and a \emph{tree} if there is a unique path between any two vertices.
A \emph{cycle} is a graph with distinct vertices $v_1,\dots,v_n$, $n\geq 3$, and edges $\set{v_i,v_{i+1}}$ for $i=1,\dots,n-1$ as well as $\set{v_1,v_n}$.
If the vertex set $V$ of a graph $G=(V,E)$ is contained in the metric space $(X,d)$, then the \emph{length} of $G$ is defined to be $\ell(G)=\sum_{\set{a,b}\in E}d(a,b)$.
Our graphs are not geometric in any sense: We do not require $a$ and $b$ to be joined by an arc in $X$ if $\set{a,b}\in E$.

For any finite non-empty set $P\subseteq X$, define
\newlength{\myl}
\settowidth{\myl}{$G=(V,E)$ is a connected graph}
\[\SMT(P)=\inf\setbuilder{\ell(G)}{\text{$G=(V,E)$ is connected with $P\subseteq V\subseteq X$}}.\]
For example, if $X$ is the Euclidean plane and $P$ the vertex set of an equilateral triangle of side length $1$, then $\SMT(P)=\sqrt{3}$, with a shortest graph connecting the vertices to the centroid of the triangle.
Since any connected $G=(V,E)$ contains a tree $T=(V,E')$ on the same vertex set (a \emph{spanning tree} of $G$), we may restrict the graphs in the definition of $\SMT(P)$ to be trees.
Trees with a vertex set that contains $P$ are called \emph{Steiner trees} on $P$.
The vertices in $V\setminus P$ (if any) are the \emph{Steiner points} of the Steiner tree.
(If the infimum in the definition of $\SMT(P)$ is attained, any tree $T$ with vertex set containing $P$ such that $\SMT(P)=\ell(T)$ is called a \emph{Steiner minimal tree} of $P$.)
Without loss of generality, we may restrict the Steiner trees in the definition of $\SMT(P)$ to the \emph{proper Steiner trees} on $P$, that is, Steiner trees on $P$ such that all Steiner points have degree at least $2$.
(We could even restrict the Steiner trees to those in which all Steiner points have degree at least $3$, but it will be important later to allow Steiner points of degree $2$.)
As observed by Gustave Choquet \cite{Choquet1938} in 1938, it is immediate that if $P$ is contained in the finite set $Q$, then $\SMT(P)\leq\SMT(Q)$.
For any $A\subseteq X$ we define its \emph{Menger--Choquet length} to be
\[ L_{MC}(A) = \sup\setbuilder{\SMT(P)}{\text{$P$ is a finite non-empty subset of $A$}}.\]
This defines a monotone function $L_{MC}\colon\CP(X)\to[0,\infty]$.
Note that $L_{MC}(A)=0$ if and only if $\card{A}\leq 1$.
Before Choquet, in 1931 Karl Menger introduced this notion for arcs in metric spaces \cite[pp.~741--742]{Menger1931} and asked whether it coincides with arc length.
Shortly after, Yukio Mimura \cite{Mimura1933} gave a proof for arcs in Euclidean space.
Choquet independently introduced $L_{MC}$ in \cite{Choquet1938} for closed subsets of Euclidean space, and announced the following result.
Unfortunately, he did not include a proof.
(See Section~\ref{history} for further historical remarks.)
\begin{choquettheorem}\label{choquet}
For any continuum $A$ in Euclidean space, \[L_{MC}(A)=L^*(A).\]
\end{choquettheorem}
In Theorem~\ref{maintheorem} below we state a more general version of Choquet's Theorem.
We first introduce a related functional, due to Menger.
Note that $L_{MC}(A)$ is not a function of the metric space $A$ on its own, but depends on the metric space $X$ in which $A$ is contained.
We can attempt to fix this by restricting the Steiner points to be in $A$.
Thus, for any finite $P\subseteq X$ and any $A\subseteq X$ define
\settowidth{\myl}{$G=(V,E)$ is a connected graph}
\[\SMT_A(P)=\inf\setbuilder{\ell(G)}{\text{$G=(V,E)$ is connected with $P\subseteq V$ and $V\setminus P\subseteq A$}}.\]
As before, we may restrict the graphs in the definition of $\SMT_A(P)$ to be Steiner trees on $P$.
Also, $\SMT_A(P)\leq\SMT_B(Q)$ whenever $P\subseteq Q$ and $A\supseteq B$.
For any $A\subseteq X$, define its \emph{intrinsic Menger length} to be
\[ L_{IM}(A) = \sup\setbuilder{\SMT_A(P)}{\text{$P$ is a finite subset of $A$}}.\]
Menger \cite[pp.~741--742]{Menger1931} introduced $L_{IM}$ for arcs in metric spaces.
Clearly, $L_{MC}(A)\leq L_{IM}(A)$ for any $A\subseteq X$.
Like $L_{MC}$, $L_{IM}$ is defined for arbitrary subsets of $X$.
For instance, if $A$ is the vertex set of a square of side length $1$, then $L_{IM}(A)=3$, while if $B=A\cup\set{c}$, where $c$ is the centre of the square, then $L_{IM}(B)=2\sqrt{2}$ (Fig.~\ref{frinkexample}).
This shows that, unlike $L_{MC}$, $L_{IM}$ is not monotone, hence not lower semicontinuous in the sense of Proposition~\ref{MClsc}.
(See Section~\ref{topology} for a detailed discussion of its continuity properties.)
Nevertheless, it follows from Choquet's Theorem that if $A$ is a continuum, then $L_{MC}(A)$ is in fact independent of the ambient space $X$.
The following generalisation of Choquet's Theorem is the main result of this paper.
\begin{theorem}\label{maintheorem}
For any subset $A$ of a metric space $X$, $L^*(A)\leq L_{MC}(A)\leq L_{IM}(A)$.
If $A$ is furthermore connected, then $L^*(A)=L_{MC}(A)=L_{IM}(A)$.
\end{theorem}
The proof, which uses nothing more than elementary properties of graphs and metric spaces, is in the next section.
As a corollary we obtain Theorem~\ref{lsc} by showing that $L_{MC}$ is lower semicontinuous on all subsets of $X$ in Lemma~\ref{steiner2} and Proposition~\ref{MClsc}.
\begin{lemma}\label{steiner2}
For all finite non-empty subsets $P$ and $Q$ of $X$, $\SMT(P)\leq\SMT(Q)+\card{P}e(P,Q)$.
\end{lemma}
\begin{proof}
Let $T_Q=(V_Q,E_Q)$ be a Steiner tree on $Q$.
For each $p\in P$, let $v(p)$ be a point in $Q$ closest to $p$.
Let $V=P\cup Q$ and $E=E_Q\cup\setbuilder{\set{p,v(p)}}{p\in P\setminus Q}$.
Then $T=(V,E)$ is a Steiner tree on $P$ of length
\[\ell(T) = \ell(T_Q)+\sum_{p\in P}d(p,v(p)) \leq \ell(T_Q)+\card{P}e(P,Q).\]
Thus, $\SMT(P)\leq\ell(T_Q)+\card{P}e(P,Q)$ for all Steiner trees $T_Q$ on $Q$.
It follows that $\SMT(P)\leq\SMT(Q)+\card{P}e(P,Q)$.
\end{proof}
\begin{proposition}\label{MClsc}
Let $A$ be a subset of a metric space $X$.
\begin{enumerate}
\item If $L_{MC}(A)<\infty$ then for each $\epsi>0$ there exists a finite collection~$\CU$ of open sets in $X$ hit by $A$ such that for any $B\subseteq X$ that hits~$\CU$, $L_{MC}(A)<L_{MC}(B)+\epsi$.
\item If $L_{MC}(A)=\infty$, then for each $M>0$ there exists a finite collection~$\CU$ of open sets in $X$ hit by $A$ such that for any $B\subseteq X$ that hits~$\CU$, $L_{MC}(B)>M$.
\end{enumerate}
\end{proposition}
\begin{proof}
Suppose first that $L_{MC}(A)<\infty$.
Let $\epsi>0$.
Choose a finite $P\subseteq A$ such that $\SMT(P) > L_{MC}(A)-\epsi/2$.
Let $\delta=\epsi/(2\card{P})$.
Then $A$ trivially hits $\CU=\setbuilder{B(p,\delta)}{p\in P}$.
Let $B\subseteq X$ be an arbitrary set that hits $\CU$, and for each $p\in P$, choose $p'\in B\cap B(p,\delta)$.
Let $P'=\setbuilder{p'}{p\in P}$.
Then $e(P,P')<\delta$, and by Lemma~\ref{steiner2}, $\SMT(P)<\SMT(P')+\card{P}\delta\leq L_{MC}(B)+\epsi/2$.
It follows that $L_{MC}(A) < L_{MC}(B)+\epsi$.

The case where $L_{MC}(A)=\infty$ can be proved by making minimal adjustments to the above proof.
\end{proof}
As another corollary of Theorem~\ref{maintheorem} we obtain a theorem of Bogn\'ar \cite{Bognar1989} (cf.\ Fremlin \cite[4A]{Fremlin1992}).
\begin{corollary}[Bogn\'ar \cite{Bognar1989}]\label{bognartheorem}
For any connected subset $A$ of a metric space $X$, $L^*(A)=L^*(\overline{A})$.
\end{corollary}
This corollary is immediate from Theorem~\ref{bognar} and the fact that the lower limit of a constant sequence $\net{A}$ is the closure $\closure{A}$.
It also follows from Theorem~\ref{maintheorem} by noting that, since $\SMT(\cdot)$ is a continuous function on the collection of finite subsets of $X$, $L_{MC}(A)=L_{MC}(\overline{A})$ for any subset $A$ of a metric space $X$.

We note in passing that if $A$ is connected and $L^*(A)<\infty$, then $A$ is $L^*$-measurable \cite[4I]{Fremlin1992}, a result which is beyond the scope of this paper.

We may ask whether the use of Steiner points is necessary.
In fact, before Menger introduced $L_{MC}$ and $L_{IM}$ for arcs in \cite{Menger1931, Menger1932}, he defined \cite[\S\ 6]{Menger1930} the length of a metric continuum $C$ in the following straightforward way:
\[ L_M(C) =\sup\setbuilder{\MST(P)}{\text{$P$ is a finite non-empty subset of $C$}},\]
where \[\MST(P)=\SMT_{\emptyset}(P)=\min\setbuilder{\ell(G)}{\text{$G=(P,E)$ is connected}}\] is the length of a minimal spanning tree on $P$.
Menger \cite{Menger1930} showed that if $C$ is an arc, $L_M(C)=L^*(C)$, and also that $L_M$ is lower semicontinuous on the collection of continua.
Unlike $L_{MC}(C)$, $L_M(C)$ depends only on $C$, and not on the metric space in which $C$ lies.
However, as pointed out by Frink \cite{Frink1936}, $L_M(C)$ in general does not necessarily equal $L^*(C)$.
To demonstrate this, Frink used Menger's own observation that $L_M$ is not monotone even if $C$ is a continuum (Fig.~\ref{frinkexample}).
\begin{figure}
\centering
\begin{tikzpicture}[thick, scale=1.5]
    \foreach \x in {-45,45,135,225} 
        {\draw[thin,gray] (0,0) -- (\x:1cm);}
    \draw (-45:1cm) \foreach \x in {45,135,225} {-- (\x:1cm)} -- cycle;
    \draw (-0.5,0) node{$1$};
    \draw (0,-1) node{$A$};
    \foreach \x in {-45,45,135,225} 
        {\fill (\x:1cm) circle (1.5pt);}
\begin{scope}[xshift=3cm]
    \draw[thin,gray] (-45:1cm) \foreach \x in {45,135,225} {-- (\x:1cm)} -- cycle;
    \foreach \x in {-45,45,135,225} 
        {\draw (0,0) -- (\x:1cm);}
    \draw (0,-1) node{$B$};
    \foreach \x in {-45,45,135,225} 
        {\fill (\x:1cm) circle (1.5pt);}
    \fill (0,0) circle (1.5pt) node[right=2mm] {$c$};
\end{scope}
\draw(1.5,0) node {$\subset$};
\draw(1.5,-1.5) node {$\SMT_A(A)=\MST(A)=3>\SMT_B(B)=\MST(B)=2\sqrt{2}$};
\end{tikzpicture}
\caption{Menger's example \cite[top of p.~476]{Menger1930} demonstrating that $P\mapsto \MST(P)$ and $P\mapsto\SMT_P(P)$ are not monotone}\label{frinkexample}
\end{figure}
Let $C$ be the union of the two diagonals of a unit square in the plane.
Then $L^*(C)=2\sqrt{2}$, but $L_M(C)=3$, as it is not too difficult to show that for any finite non-empty $P\subseteq C$, $\MST(P)\leq 3$ with equality if and only if $P=A$, the set consisting of the four vertices of the square.
(Note that $\SMT(A)=1+\sqrt{3} < 2\sqrt{2}=\MST(B)$, where $B=A\cup\set{c}$, with $c$ the centre of the square.)

Nevertheless, it follows from Lemma~\ref{moore} below that \[L_{MC}(A)\leq L_{IM}(A)\leq L_M(A)\leq 2L_{MC}(A)\] for any $A\subseteq X$, and therefore, if any one of these quantities is finite, all of them are.
There is a simple characterization of metric spaces $A$ for which any of these quantities is finite.
They are exactly the subsets of metric continua of finite outer linear measure.
\begin{theorem}\label{shortest1}
Any metric space $A$ with $L_{MC}(A)<\infty$ can be embedded into a metric continuum $C$ such that $L_{MC}(A)=L_{MC}(C)=L^*(C)$.
\end{theorem}
The proof, although outside the scope of this paper, uses Gromov's compactness theorem for Gromov--Hausdorff convergence \cite{Ducret2008, PaoliniStepanov2013}.
We also mention the following result that can be proved in a standard way with the help of the Blaschke selection theorem for Hausdorff convergence.
As mentioned before, a metric space $X$ is \emph{convex} if any two points $a,b\in X$ are joined by an arc of length $d(a,b)$.
A metric space is called \emph{proper} if all closed and bounded sets are compact.
\begin{theorem}\label{shortest2}
Let $X$ be a convex, proper metric space.
For any $A\subseteq X$, if $L_{MC}(A)<\infty$ then $A$ is contained in a continuum $C\subseteq X$ such that $L_{MC}(A)=L_{MC}(C)=L^*(C)$.
\end{theorem}
In Section~\ref{history} we discuss more general results on the existence of connected sets of minimum length containing a given set.
\section{Proof of Theorem~\ref{maintheorem}}\label{proof}
Recall that a Steiner tree is called proper if the degree of each Steiner point is at least $2$.
As mentioned before, it is very important for our purposes to allow Steiner points of degree $2$.
However, as is well known, in any Steiner tree on a set $P$, there are at most $\card{P}-2$ Steiner points of degree at least $3$.
This follows from the fact that a tree on $\card{V}$ vertices has $\card{V}-1$ edges and degree counting.
A \emph{chain} in a Steiner tree on $P$ is a path in the tree such that all vertices of the path, except possibly the end-vertices, are Steiner points of degree $2$.
A chain in a Steiner tree is called \emph{maximal} if it is not properly contained in a larger chain.
Let $T$ be a proper Steiner tree.
Each vertex of $T$ is contained in some maximal chain of $T$, and each edge of $T$ is contained in a unique maximal chain.
Thus, any two maximal chains are edge-disjoint.
Also, if two maximal chains have a common vertex, this vertex must be a common endpoint that is not a Steiner point of degree $2$.
(See Fig.~\ref{chain}.)
\begin{figure}
\centering
\begin{tikzpicture}[scale=1.5]
\draw (0,0) -- (0.25,0.30) -- (0.5,0.55)  -- (0.78,0.72)  -- (1,1)  -- (1.3,0.87)  -- (1.7,0.9)  -- (2,1)  -- (2.30,0.75)  -- (2.60,0.5)  -- (2.80,0.25)  -- (3,0) ;
\draw (0,2)  -- (0.20,1.75)  -- (0.53,1.5)  -- (0.70,1.25)  -- (1,1);
\draw (2,1) -- (2.25,1.30)  -- (2.5,1.52)  -- (2.75,1.80)  -- (3,2) ;

\fill (0,0) circle (1.5pt);
\fill (0.25,0.30) circle (1pt);
\fill (0.5,0.55) circle (1pt);
\fill (0.78,0.72) circle (1pt);

\fill (1,1) circle (1pt);

\fill (1.3,0.87) circle (1pt);
\fill (1.7,0.9) circle (1pt);

\fill (2,1) circle (1pt);

\fill (2.30,0.75) circle (1pt);
\fill (2.60,0.5) circle (1pt);
\fill (2.80,0.25) circle (1pt);
\fill (3,0) circle (1.5pt);

\fill (0,2) circle (1.5pt);
\fill (0.20,1.75) circle (1pt);
\fill (0.53,1.5) circle (1pt);
\fill (0.70,1.25) circle (1pt);

\fill (2.25,1.30) circle (1pt);
\fill (2.5,1.52) circle (1pt);
\fill (2.75,1.80) circle (1pt);
\fill (3,2) circle (1.5pt);
%%%%%%%%%%%%%%%%%%%%%%%%%%%%%%%%%%%%%%%

\draw[thick, blue] (-0.2,0) .. controls (-0.3,0) and (0.3,0.5) .. node[above left] {$C_4$}
(0.8,1) .. controls (0.9,1.1) and (1.1,1.2) .. 
(1.15,1) .. controls (1.3,0.9) and (0.7,0.5) .. 
(0.2,0) .. controls (0.1,-0.1) and (-0.15,-0.25) .. (-0.2,0) -- cycle;

\draw[thick, blue] (-0.2,2) .. controls (-0.3,2) and (0.3,1.5) .. node[below left] {$C_1$}
(0.8,1) .. controls (0.9,0.9) and (1.1,0.8) .. 
(1.15,1) .. controls (1.3,1.1) and (0.7,1.5) .. 
(0.2,2) .. controls (0.1,2.1) and (-0.15,2.25) .. (-0.2,2) -- cycle;

\draw[thick, blue] (3.2,0) .. controls (3.3,0) and (2.9,0.5) .. node[above right] {$C_5$}
(2.2,1) .. controls (2.1,1.1) and (1.9,1.2) .. 
(1.85,1) .. controls (1.7,0.9) and (2.5,0.5) .. 
(2.8,0) .. controls (2.9,-0.1) and (3.15,-0.25) .. (3.2,0) -- cycle;

\draw[thick, blue] (3.2,2) .. controls (3.3,2) and (2.7,1.5) .. node[below right] {$C_3$}
(2.2,1) .. controls (2.1,0.9) and (1.9,0.8) .. 
(1.85,1) .. controls (1.7,1.1) and (2.3,1.5) .. 
(2.8,2) .. controls (2.9,2.1) and (3.15,2.25) .. (3.2,2) -- cycle;

\draw[thick, blue] (1,0.8) .. controls (0.8,0.8) and (0.8,1.1) ..
(1.05,1.16) .. controls (1.3,0.9) and (1.7,0.95) .. node[above] {$C_2$}
(2,1.16) .. controls (2.2,1.2) and (2.2,0.8) ..
(2,0.8) .. controls (1.7,0.7) and (1.3,0.7) .. (1,0.8) -- cycle;

\end{tikzpicture}
\caption{A proper Steiner tree on a set of $4$ points can be covered by $5$ maximal chains.}\label{chain}
\end{figure}
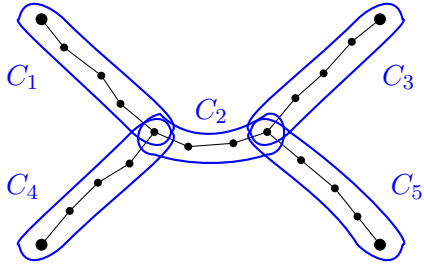
We can replace each maximal chain in a proper Steiner tree $T$ by a single edge joining the endpoints of the chain to obtain a \emph{reduced} Steiner tree $T'$ on the same set where each Steiner point has degree at least $3$, and such that $\ell(T')\leq \ell(T)$ (by the triangle inequality).
\begin{lemma}\label{chain1}
A proper Steiner tree $T=(V,E)$ on a finite set $P$ has at most $2\card{P}-3$ maximal chains.
\end{lemma}
\begin{proof}
This follows from the above discussion by noting that in a reduced Steiner tree there are at most $\card{P}-2$ Steiner points, hence there are at most $2\card{P}-2$ vertices in total, hence there are at most $2\card{P}-3$ edges.
\end{proof}
The following is a well-known result occurring implicitly or explicitly in \cite{Wazewski1927, Kolmogorov1933, Milgram1940, Eilenberg1943, Gilbert-Pollak1968}.
\begin{lemma}\label{moore2}
Let $P$ be a finite subset of a metric space $X$ with $\card{P}\geq 3$.
Then for any Steiner tree $T$ on $P$, there exists a cycle $C$ with vertex set $P$ such that $\ell(C)\leq 2\ell(T)$.
\end{lemma}
\begin{proof}
By possible removing edges from $T$, we may assume that $T$ is proper.
Also, we may replace $T$ by a reduced Steiner tree without increasing its length.
It is therefore sufficient to show the lemma for reduced Steiner trees on $P$.
We use induction on $\card{P}\geq 3$.
For the case $\card{P}=3$, note that there are two possible reduced Steiner trees on $P$, one with no Steiner points, and one with a single Steiner point joined to each point in $P$.
In both cases, the inequality follows from the triangle inequality.

In the case $\card{P}>3$, if we remove each point in $P$ of degree $1$ from $T$, the remaining graph is still a tree, hence has a vertex $v$ of degree $1$.
This vertex is necessarily a Steiner point, hence is joined to two points $p_1,p_2\in P$.
Then the tree $T'$ obtained from $T$ by removing the points $p_1$, $p_1$, and the edges $\set{p_1,v}$ and $\set{p_2,v}$ is a Steiner tree on $P'=P\cup\set{v}\setminus\set{p_1,p_2}$, and by induction, there is a cycle $C'$ on $P'$ such that $\ell(C')\leq2\ell(T')$.
Let $C$ be the cycle on $P$ obtained by replacing the two edges $\set{v,u}$ and $\set{v,w}$ incident to $v$ by $\set{u,p_1}$, $\set{p_1,p_2}$, $\set{p_2,w}$.
Then it follows from the triangle inequality that $\ell(C)\leq 2\ell(T)$.
\end{proof}
\begin{lemma}\label{moore}
For any finite non-empty $P\subseteq X$, \[\MST(P)\leq \frac{2(\card{P}-1)}{\card{P}}\SMT(P).\]
\end{lemma}
\begin{proof}
The lemma is trivial if $\card{P}=1$, and follows from the triangle inequality for $\card{P}=2$.
Let $\epsi>0$.
Let $T$ be a Steiner tree on $P$ with $\ell(T) < \SMT(P)+\epsi$.
By Lemma~\ref{moore2} there exists a cycle through $P$ of length at most $2\ell(T)$.
By removing the longest edge from this cycle we obtain a path through $P$ of length at most $\frac{\card{P}-1}{\card{P}}2\ell(T)$, that is,
\[ \MST(P)\leq \frac{\card{P}-1}{\card{P}}2\ell(T) < \frac{\card{P}-1}{\card{P}}2(\SMT(P)+\epsi).\]
Since this holds for any $\epsi>0$, we obtain the conclusion.
\end{proof}

Let $S\subseteq A\subseteq X$ and $\epsi>0$.
Then $S$ is called \emph{$\epsi$-separated} if $d(x,y)\geq\epsi$ for all distinct $x,y\in S$.
Also, $S$ is called an \emph{$\epsi$-net} of $A$ if each point of $A$ has distance $<\epsi$ to some point of $S$.
Note that a maximal $\epsi$-separated subset of $A$ is also an $\epsi$-net of $A$.
The set $A$ is said to be \emph{totally bounded} if $A$ contains a finite $\epsi$-net for each $\epsi>0$.
\begin{lemma}\label{totallybounded}
Let $A\subseteq X$ satisfy $L_{MC}(A)<\infty$.
Then for any $\epsi>0$, any $\epsi$-separated subset of $A$ has at most $\max\set{\frac{2}{\epsi}L_{MC}(A),1}$ points.
In particular, $A$ is totally bounded.
\end{lemma}
\begin{proof}
Let $P$ be a finite $\epsi$-separated subset of $A$, that is, $d(x,y)\geq\epsi$ for all distinct $x,y\in P$.
Without loss of generality, $\card{P}\geq 2$.
Then $\MST(P)\geq (\card{P}-1)\epsi$, since a spanning tree of $P$ has $\card{P}-1$ edges, and by the previous lemma,
\[\SMT(P) \geq \frac{\card{P}}{2(\card{P}-1)}\MST(P)
\geq \frac{\card{P}}{2(\card{P}-1)}(\card{P}-1)\epsi = \frac{\epsi}{2}\card{P}.\]
It follows that \[\card{P}\leq \frac{2}{\epsi}\SMT(P)\leq\frac{2}{\epsi}L_{MC}(A).\]
Therefore, if we choose points of $A$ one by one, keeping the chosen subset $\epsi$-separated, we will stop after at most $\max\set{1,\frac{2}{\epsi}L_{MC}(A)}$ steps with a maximal $\epsi$-separated subset of $A$.
\end{proof}

\begin{lemma}\label{chain2}
Consider a chain $C=x_1x_2\dots x_n$ of length $\ell(C)$ in a Steiner tree.
Then for any $t>0$ there exist $k < 1+2\ell(C)/t$ vertex-disjoint subchains $C_1,\dots,C_k$ such that each $\ell(C_i)\leq t$, the vertex set of $C$ is partitioned by the vertex sets of the $C_i$ and all edges of $C$ not in any $C_i$ have length $>t$.
\end{lemma}
\begin{proof}
We use the following simple online bin-packing algorithm.
Walk along $C$ from $x_1$ to $x_n$, cutting it up at vertices into chains that are as long as possible, trying to keep their lengths $\leq t$.
Whenever an edge of length $>t$ is found, it is taken to be a single ``outsize'' chain (Fig.~\ref{fig:chain}).
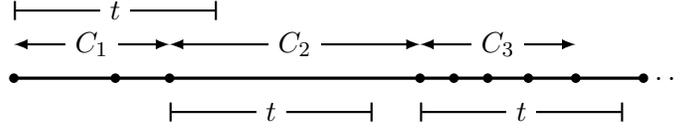
\begin{figure}
\centering
\begin{tikzpicture}[thick,scale=0.9]
\draw[|-|] (0,1) -- node[midway,fill=white] {$t$} (3,1);
\draw[|-|] (2.3,-0.5) -- node[midway,fill=white] {$t$} (5.3,-0.5);
\draw[|-|] (6,-0.5) -- node[midway,fill=white] {$t$} (9,-0.5);
\draw[very thick] (0,0) -- (9.3,0) node [right] {\dots};
\fill (0,0) circle (2pt);
\fill (1.5,0) circle (2pt);
\fill (2.3,0) circle (2pt);
\draw[latex-latex] (0,0.5) -- node[midway,fill=white] {$C_1$} (2.3,0.5);
\fill (6,0) circle (2pt);
\fill (6.5,0) circle (2pt);
\fill (7,0) circle (2pt);
\fill (7.6,0) circle (2pt);
\fill (8.3,0) circle (2pt);
\draw[latex-latex] (2.3,0.5) -- node[midway,fill=white] {$C_2$} (6,0.5);
\draw[latex-latex] (6,0.5) -- node[midway,fill=white] {$C_3$} (8.3,0.5);
\fill (9.3,0) circle (2pt);

\end{tikzpicture}
\caption{Online chain cutting}\label{fig:chain}
\end{figure}
Denote the consecutive chains made in this way by $C_1,\dots,C_k$.
Because of maximality, the sum of the lengths of any two adjacent chains is $>t$.
It follows that $2\ell(C)>(k-1)t$.
Finally, remove all outsize chains and add a new chain $x_i$ of length $0$ between any two adjacent outsize chains $x_{i-1}x_i$ and $x_ix_{i+1}$.
In this way we remove at least as many chains as we add new chains of length $0$.
We end up with at most $k$ chains, each of length $\leq t$, with vertices covering the vertices of $C$.
\end{proof}
\begin{proof}[Proof of Theorem~\ref{maintheorem}]
We first show that $L^*(A)\leq L_{MC}(A)$.
Without loss of generality, $L_{MC}(A)<\infty$.
Then, by Lemma~\ref{totallybounded}, $A$ is totally bounded.
The idea of the proof is as follows.
We take a finite $P\subseteq A$ with $\SMT(P)$ very close to $L_{MC}(A)$.
Then we extend $P$ to an $\epsi$-net $P'$ with $\epsi$ very small, depending on $\delta$ and another auxiliary large constant $R$.
Then $\SMT(P')$ is still very close to $L_{MC}(A)$.
We take a proper Steiner tree $T'$ on $P'$ of length very close to $\SMT(P')$, and use Lemmas~\ref{chain1} and \ref{chain2} to cut the proper subtree $T$ of $T'$ joining $P$ into small pieces to create a good $\delta$-cover of $A$.

Let $\delta>0$ be arbitrary.
Without loss of generality, we may assume $\delta<1/8$.
Fix a finite $P\subseteq A$ with $\SMT(P)>L_{MC}(A)-\delta/4$.
Set
\[ \epsi = \min \set{ \min_{x,y\in P,x\neq y} d(x,y), \delta^2,\frac{\delta}{\card{P}} }. \]
Then $P$ is $\epsi$-separated, and by Lemma~\ref{totallybounded}, $P$ is contained in a maximal $\epsi$-separated finite $P'\subseteq A$ which is also an $\epsi$-net.
Let $T'=(V',E')$ be a proper Steiner tree on $P'$ such that
\[ \ell(T') < \SMT(P') +\frac{\delta}{4}\leq L_{MC}(A)+\frac{\delta}{4}.\]
We decompose $T'$ as follows.
Let $T=(V,E)$ be the union of all the paths in $T'$ between pairs of points from $P$.
Then $T$ is clearly a Steiner tree on $P$.
Remove the edges of $T$ from $T'$ (but not the vertices)
to obtain $G=(V',E'\setminus E)$.
Let $T_v$ be the connected component of $G$ that contains $v\in V$.
Denote the maximal chains of $T$ by $C_i$, $i=1,\dots,c$ (see Fig.~\ref{decomposition}).
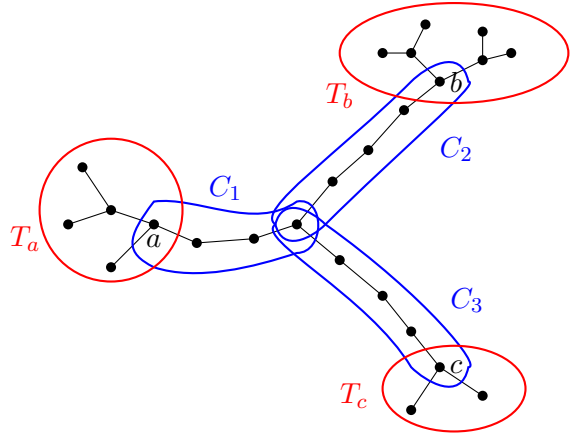
\begin{figure}
\centering
\begin{tikzpicture}[scale=1.9]
\draw (1,1) -- (0.7,1.1) -- (0.5,1.4);
\draw (0.7,1.1) -- (0.4,1);
\draw (0.7,0.7) -- (1,1)  -- (1.3,0.87)  -- (1.7,0.9)  -- (2,1)  -- (2.30,0.75)  -- (2.60,0.5)  -- (2.80,0.25)  -- (3,0) -- (2.8,-0.3);
\draw (3,0) -- (3.3,-0.2);
%\draw (0,2)  -- (0.20,1.75)  -- (0.53,1.5)  -- (0.70,1.25)  -- (1,1);
\draw (2,1) -- (2.25,1.30)  -- (2.5,1.52)  -- (2.75,1.80)  -- (3,2) -- (2.8,2.2) -- (2.9,2.4);
\draw (2.8,2.2) -- (2.6,2.2);
\draw (3,2) -- (3.3,2.15) -- (3.5,2.2);
\draw (3.3,2.15) -- (3.3,2.35);

%\fill (0,0) circle (1.5pt);
%\fill (0.25,0.30) circle (1pt);
%\fill (0.5,0.55) circle (1pt);
%\fill (0.78,0.72) circle (1pt);

\fill (1,1) circle (1pt) node[below] {$a$};
\fill (0.7,1.1) circle (1pt);
\fill (0.5,1.4) circle (1pt);
\fill (0.4,1) circle (1pt);
\fill (0.7,0.7) circle (1pt);

\fill (1.3,0.87) circle (1pt);
\fill (1.7,0.9) circle (1pt);

\fill (2,1) circle (1pt);

\fill (2.30,0.75) circle (1pt);
\fill (2.60,0.5) circle (1pt);
\fill (2.80,0.25) circle (1pt);
\fill (3,0) circle (1pt) node[right] {$c$};
\fill (2.8,-0.3) circle (1pt);
\fill (3.3,-0.2) circle (1pt);

\fill (2.25,1.30) circle (1pt);
\fill (2.5,1.52) circle (1pt);
\fill (2.75,1.80) circle (1pt);
\fill (3,2) circle (1pt) node[right] {$b$};
\fill (2.8,2.2) circle (1pt);
\fill (2.9,2.4) circle (1pt);
\fill (2.6,2.2) circle (1pt);
\fill (3.3,2.15) circle (1pt);
\fill (3.5,2.2) circle (1pt);
\fill (3.3,2.35) circle (1pt);
%%%%%%%%%%%%%%%%%%%%%%%%%%%%%%%%%%%%%%%

\draw[thick, blue] (3.2,0) .. controls (3.3,0) and (2.9,0.5) .. node[above right] {$C_3$}
(2.2,1) .. controls (2.1,1.1) and (1.9,1.2) .. 
(1.85,1) .. controls (1.7,0.9) and (2.5,0.5) .. 
(2.8,0) .. controls (2.9,-0.1) and (3.15,-0.25) .. (3.2,0) -- cycle;

\draw[thick, blue] (3.2,2) .. controls (3.3,2) and (2.7,1.5) .. node[below right] {$C_2$}
(2.2,1) .. controls (2.1,0.9) and (1.9,0.8) .. 
(1.85,1) .. controls (1.7,1.1) and (2.3,1.5) .. 
(2.8,2) .. controls (2.9,2.1) and (3.15,2.25) .. (3.2,2) -- cycle;

\draw[thick, blue] (1,0.7) .. controls (0.8,0.8) and (0.8,0.9) ..
(1,1.16) .. controls (1.3,1.2) and (1.7,0.95) .. node[above] {$C_1$}
(2,1.16) .. controls (2.2,1.2) and (2.2,0.8) ..
(2,0.8) .. controls (1.7,0.7) and (1.3,0.6) .. (1,0.7) -- cycle;

\draw[thick, red] (0.7,1.1) circle (5mm); \draw[red] (0.1,0.9) node {$T_a$};
\draw[thick, red] (3.1,2.2) ellipse (8mm and 3.5mm); \draw[red] (2.3,1.9) node {$T_b$};
\draw[thick, red] (3.1,-0.15) ellipse (5mm and 3mm); \draw[red] (2.4,-0.2) node {$T_c$};

\end{tikzpicture}
\caption{Decomposing the Steiner tree $T'$ with $P=\set{a,b,c}$}\label{decomposition}
\end{figure}
By Lemma~\ref{chain1}, $c\leq2\card{P}-3$.
We use Lemma~\ref{chain2} to cut each chain $C_i$ up into disjoint subchains $C_{i,1},\dots,C_{i,k(i)}$, each of length $\ell(C_{i,j})\leq \delta/2$, where the number of pieces for each $C_i$ is $k(i)< 1+\frac{4\ell(C_i)}{\delta}$.
Then the $C_{i,j}$ and the $T_v$ together cover the vertex set of $T'$, which includes the $\epsi$-net $P'$.

Some of the $T_v$ are so small that they may be ignored, as follows.
We say that $T_v$ is \emph{large} if $\ell(T_v)\geq\epsi$.
Let $L=\setbuilder{v\in V}{\ell(T_v)\geq\epsi}$ be the set of vertices $v$ with a large $T_v$.
Let $a\in A$.
Then there exists $p\in P'\subseteq V'$ with $d(a,p)<\epsi$.
If $p$ is not in any chain $C_i$ or large $T_v$, then $p$ is contained in some $T_v$ of length $\ell(T_v)<\epsi$.
However, then $d(p,v)<\epsi$, hence $d(a,v)<2\epsi$ for some $v\in V$.
Therefore, \[U_{i,j}=\bigcup_{v\in V(C_{i,j})} B(v,2\epsi)\qquad (1\leq i\leq c, 1\leq j\leq k(i))\]
and \[U_v=\bigcup_{w\in V(T_v)} B(w,\epsi)\qquad(v\in L)\]
together cover $A$.
We next estimate the diameters of the $U_{i,j}$ and $U_v$ in this cover.
First, 
\[\diam(U_{i,j}) \leq \diam(C_{i,j})+4\epsi\leq \ell(C_{i,j})+4\epsi\leq\frac{\delta}{2}+4\delta^2 <\delta.\]
Next, the total length of the large $T_v$ is
\begin{align*}
\sum_{v\in L}\ell(T_v) &\leq \ell(T')-\ell(T)\leq\ell(T')-\SMT(P)\\ &< \left(L_{MC}(A)+\frac{\delta}{4}\right) - \left(L_{MC}(A)-\frac{\delta}{4}\right)= \frac{\delta}{2}.
\end{align*}
In particular, since $\card{L}\epsi\leq\sum_{v\in L}\ell(T_v)$, the number of large $T_v$ is $\card{L} < \delta/(2\epsi)$.
Since each $T_v$ is connected, $U_v$ has diameter
\begin{align}
\diam(U_v)&\leq\ell(T_v)+2\epsi < \frac{\delta}{2}+2\delta^2 < \delta. \label{B}
\end{align}
It follows that
\[ \setbuilder{U_{i,j}}{1\leq i\leq c, 1\leq j\leq k(i)}\cup\setbuilder{U_v}{v\in L}\]
is a $\delta$-cover of $A$ with
\begin{align*}
&\phantom{{}\mathrel{<}{}} \sum_{i=1}^c\sum_{j=1}^{k(i)}\diam(U_{i,j})+\sum_{v\in L}\diam(U_v)\\
&< \sum_{i=1}^c\sum_{j=1}^{k(i)}(\ell(C_{i,j})+4\epsi)+\sum_{v\in L}(\ell(T_v)+2\epsi)\\
&< \ell(T)+4\epsi\sum_{i=1}^c\left(1+\frac{4\ell(C_i)}{\delta}\right)+2\epsi\card{L}\\
&\leq \left(1+\frac{16\epsi}{\delta}\right)\ell(T) + 4 \epsi c +2\epsi\card{L}\\
&\leq \left(1+16\delta\right)\ell(T') + 4 (2\card{P}-3)\epsi +\delta\\
&< \left(1+16\delta\right)\left(L_{MC}(A)+\frac{\delta}{4}\right)+8\delta+\delta.
\end{align*}
Therefore, $L^*_\delta(A)\leq \left(1+16\delta\right)\left(L_{MC}(A)+\frac{\delta}{4}\right)+9\delta$.
Letting $\delta\to0$, we obtain \[L^*(A)=\lim_{\delta\to 0} L^*_\delta(A)\leq L_{MC}(A).\]

The second inequality $L_{MC}(A)\leq L_{IM}(A)$ is obvious.

\smallskip
It remains to show that if $A$ is a connected subset of $X$, then $L_{IM}(A)\leq L^*(A)$.
We may assume without loss of generality that $L^*(A)<\infty$.
As already remarked, since neither $L_{IM}$ nor $L^*$ depends on $X$, we may also assume that $A$ is the whole metric space.
Let $\epsi>0$ and let $P$ be an arbitrary non-empty finite subset of $A$.
We want to use a good $\delta$-covering of $A$ (for sufficiently small $\delta>0$) to construct a good Steiner tree of $P$ with Steiner points in $A$, in order to show that $\SMT_A(P)\leq L^*(A)+\epsi$.
Without loss of generality, $\card{P}\geq 2$.
Since $A$ then has at least two points, its connectedness implies that any non-empty open subset of $A$ is infinite.
Let \[\delta=\min\set{\frac{\epsi}{2\card{P}}, \min_{x,y\in P, x\neq y}d(x,y)}.\]
Let $\setbuilder{U_i}{i\in\bN}$ be a $\delta$-cover of $A$ such that
\[ \sum_{i=1}^\infty\diam(U_i) < L^*(A)+\frac{\epsi}{2}.\]
If we instead take a $(\delta/2)$-cover and replace each $U_i$ by $\bigcup_{x\in U_i}B(x,\eta/2^i)$ for a sufficiently small $\eta>0$, we may assume that each $U_i$ is open in $A$.
By the choice of $\delta$, $\card{U_i\cap P}\leq 1$ for all $i\in\bN$.
We say that two points $x,y\in A$ are \emph{joined by $\setbuilder{U_i}{i\in\bN}$} if there exists a finite sequence $U_{i(1)},\dots,U_{i(k)}$ such that $x\in U_{i(1)}$, $y\in U_{i(k)}$, and for each $t=1,\dots,k-1$, $U_{i(t)}\cap U_{i(t+1)}\neq\emptyset$.
The connectedness of $A$ implies that any two $x,y\in A$ are connected by $\set{U_i}$.
Indeed, for a fixed $x\in A$, the set $S$ of all $y\in A$ such that $x$ and $y$ are connected by $\set{U_i}$ is open, as it is a union of $U_i$.
The complement $A\setminus S$ is also open, since if $z\in A\setminus S$, then $z\in U_j$ for some $j$ and then $U_j\subseteq A\setminus S$, otherwise there would exist $y\in U_j\cap S$, hence the sequence $U_{i(1)},\dots,U_{i(k)}$ demonstrating that $x$ and $y$ are joined, together with $U_{i(k+1)}=U_j$, show that $z\in S$, a contradiction.

For any $x,y\in A$, choose a sequence $U_{i(1,x,y)},\dots,U_{i(k(x,y),x,y)}$ demonstrating that $x$ and $y$ are joined, where we choose a single $U_{i(1,x,x)}$ when $x=y$.
Then the graph $\CG$ on the set $\CV=\setbuilder{i(j,x,y)}{x,y\in A, 1\leq j\leq k(x,y)}$ with edges $\CE=\setbuilder{\set{i,j}}{U_i\cap U_j\neq\emptyset}$ is connected, hence contains a minimal subgraph that contains $\setbuilder{i(1,x,x)}{x\in A}$.
This subgraph is necessarily a tree $\CT=(I,\CE')$, with $I$ a non-empty finite subset of $\bN$ and $\CE'$ a finite subset of $\CE$.
For each edge $e=\set{i,j}\in\CE'$, choose $x_{e}\in (U_i\cap U_j)\setminus P$ (recall that $U_i\cap U_j$ is infinite).
For each $i\in I$, let $V_i=(P\cap U_i)\cup\setbuilder{x_{e}}{i\in e\in\CE'}$.
Then
\begin{equation}\label{star}
\card{V_i} \leq \begin{cases}
\deg_\CT(i) &\text{if $P\cap U_i=\emptyset$,}\\
\deg_\CT(i)+1 &\text{if $P\cap U_i\neq\emptyset$.}
\end{cases}
\end{equation}
Let $i\in I$.
Since $\deg_\CT(i)\geq 1$, $\card{V_i}\geq 1$.
If $\card{V_i}=1$, then $\deg_\CT(i)=1$ and $P\cap U_i=\emptyset$.
However, then $i$ and the edge of $\CT$ incident with $i$ are redundant, contradicting the minimality of $\CT$.
Therefore, $\card{V_i}\geq 2$ for each $i\in I$.
Form a graph $G$ on the vertex set $V=\bigcup_{i\in I}V_i$ by joining the vertices in each $V_i$ by an arbitrary path.
Since $\CT$ is connected, $G$ is also connected.
Note that two $V_i$ could intersect in more than one point, hence $G$ is not necessarily a tree.
Nevertheless, any spanning tree $T$ of $G$ will be a Steiner tree on $P$, and we obtain
\begin{align*}
\ell(T)\leq \ell(G) &\leq \sum_{i\in I}(\card{V_i}-1)\diam(U_i)\\
&= \sum_{i\in I}\diam(U_i) + \sum_{i\in I}(\card{V_i}-2)\diam(U_i)\\
&\leq L^*(A)+\frac{\epsi}{2} + \delta\sum_{i\in I}(\card{V_i}-2).
\end{align*}
By \eqref{star} and $\card{\CE'}=\card{I}-1$,
\begin{equation*}
\sum_{i\in I}(\card{V_i}-2) \leq \sum_{i\in I}(\deg_\CT(i)-2)+\card{P}
= -2+\card{P}.
\end{equation*}
It follows that
\begin{equation*}
\SMT_A(P) \leq \ell(T) \leq L^*(A) +\frac{\epsi}{2} + (\card{P}-2)\delta
\leq L^*(A) +\epsi.
\end{equation*}
This holds for arbitrary $P\subseteq A$ and $\epsi>0$, and we conclude that $L_{IM}(A)\leq L^*(A)$.
\end{proof}
\section{Lower semicontinuous with respect to which topology?}\label{topology}
Given a topological space $(\CX,\CT)$, a function $f\colon \CX\to\bR$ is defined to be \emph{lower semicontinuous} at a point $x_0\in \CX$ if for any $\epsi>0$ there exists an open neighbourhood $U$ of $x_0$ such that for all $x\in U$, $f(x_0)<f(x)+\epsi$.
Go\l\k{a}b's original theorem and each of the Theorems~\ref{golab}, \ref{frink}, \ref{bognar}, \ref{fremlin}, \ref{lsc} assert the lower semicontinuity of $L^*$ on the collection of connected subsets of a metric space for an appropriate topology on this collection.
For instance, the well-known \emph{Hausdorff topology} is the topology on the collection $\CC(X)$ of all closed subsets of $X$ generated by the basis of all sets of the form
\[ \CB(A,r) = \setbuilder{S\in\CC(X)}{D(A,S)<r}, \qquad A\in\CC(X), r>0.\]
Go\l\k{a}b's theorem states that on the subspace of all continua of $X$, $L^*$ is lower semicontinuous with respect to the Hausdorff topology.
Theorem~\ref{golab} can also be stated as the lower semicontinuity of $L^*$, but now on the coarser \emph{lower Hausdorff topology}, which we define on the space $\CP(X)$ of all subsets of $X$.
This is the topology $\CH^-$ generated by the basis consisting of all sets of the form
\[ \CB^-(A,r) = \setbuilder{S\in\CP(X)}{e(A,S)<r}, \qquad A\in\CP(X), r>0.\]
Then Theorems~\ref{golab} and \ref{frink} assert the lower semicontinuity of $L^*$ on the subspace of $(\CP(X),\CH^-)$ consisting of the connected subsets of $X$.

For Theorem~\ref{bognar}, \ref{fremlin} and \ref{lsc}, we introduce the
\emph{lower Vietoris topology}.
This is the topology $\CV^-$ on $\CP(X)$ generated by the subbasis consisting of all sets of the form
\[ A^- = \setbuilder{S\in\CP(X)}{S\cap A\neq\emptyset}, \qquad A\text{ open in }X.\]
Theorems~\ref{bognar}, \ref{fremlin} and \ref{lsc}  assert the lower semicontinuity of $L^*$ on the subspace of $(\CP(X),\CV^-)$ consisting of the connected subsets of $X$.
Since the lower Vietoris topology is coarser than the lower Hausdorff topology \cite[Lemma~3.2]{Michael1951} (see also \cite[Prop.~4.2.1(ii)]{Klein1984}), Theorems~\ref{bognar}, \ref{fremlin} and \ref{lsc} are formally stronger than Theorems~\ref{golab} and \ref{frink}.
However, the difference is slight, as these two topologies coincide if and only if the metric space $X$ is totally bounded \cite[Theorem~5]{Levi1993}, and for general metric spaces $X$ the two topologies coincide on the subset of $\CP(X)$ of all totally bounded subsets of $X$.
(This can be proved analogously to Prop.~4.2.2(ii) of \cite{Klein1984}.)
Since sets that are not totally bounded have infinite outer linear measure, it follows that at points of $\CP(X)$ that have finite outer linear measure, the two topologies are the same.
Therefore, the first part of Theorem~\ref{lsc} has the same strength as the first part of Theorem~\ref{frink}.

\section{Historical remarks}\label{history}
\subsection{The length of continua}
The definition of outer linear measure was originally introduced by Carath\'eodory \cite{Caratheodory1914} and generalised by Hausdorff \cite{Hausdorff1918} to arbitrary (including fractional) dimensions.

Menger \cite{Menger1931, Menger1932}, Mimura \cite{Mimura1933} and Choquet \cite{Choquet1938} were early pioneers in the study of Steiner trees (see the historical surveys \cite{Schrijver2005} and \cite{BGTZ}).
No doubt unaware of Choquet's paper, Leonard Blumenthal included Menger's question whether $L_{MC}$ equals arc length in general metric spaces in his book on distance geometry \cite[pp.~66--67]{BlumenthalBook}.
His student, William Ettling \cite{Ettling1978}, claimed to have a proof for all proper metric spaces, but unfortunately his proof, which follows Mimura's first proof very closely,
made the tacit assumption that $X$ is \emph{convex} in the sense that any two points $x,y\in X$ are joined by an arc of length $d(x,y)$.
It is also not difficult to modify Mimura's first proof to work for any connected subset of any metric space, without assuming compactness of the set or convexity of the space, or having to use Go\l\k{a}b's Theorem.
(Unfortunately, Mimura's second proof uses the lower semicontinuity of the length of continua.)

\subsection{Existence of connected sets of minimum length}
For further existence results along the lines of Theorems~\ref{shortest1} and \ref{shortest2}, see Chapter~4 of the book of Ambrosio and Tilli \cite{Ambrosio2004}, and the papers of Ducret and Troyanov \cite{Ducret2008},
Paolini and Ulivi \cite{PaoliniUlivi2010} and 
Paolini and Stepanov \cite{PaoliniStepanov2013}.
Ivanov, Nikonov and Tuzhilin \cite{Ivanov2005} characterize the countable metric spaces that have an infinite spanning tree of finite total length.
Starting with Jones \cite{Jones1990}, a different approach has been taken to characterize the subsets of Euclidean spaces that are contained in rectifiable curves.
For an overview of this approach, see Schul's survey \cite{Schul2007}.

\subsection{Go\l\k{a}b's Theorem and its proofs}
Go\l\k{a}b's own proof \cite{Golab1928} of his theorem depends on various results of Wa\.zewski \cite{Wazewski1927}, one of them stating that a continuum $C$ with $L^*(C)<\infty$ in $\bR^n$ is the image of an absolutely continuous mapping $f\colon[0,L^*(C)]\to\bR^n$ that is non-expansive in each coordinate.
This is related to Lemma~\ref{moore2} (see also \cite[Theorem~2]{Eilenberg1943} and \cite[Theorem~4.4]{Alberti2017}).
His proof furthermore uses compactness in the form of the Arzela--Ascoli Theorem.
As can be seen from the proof in this paper, as well as Frink's proof \cite{Frink1936}, it is not necessary to use compactness or to take subsequences.
Among all proofs of Go\l\k{a}b's Theorem known to us, Frink's proof seems to be the simplest and shortest, and proceeds directly from the definition of outer linear measure.

At a point where Go\l\k{a}b's Theorem would be needed, Besicovitch \cite[Proof of Theorem~12]{Besicovitch1938} merely states that it is ``easy to see\dots''.

Faber, Mycielski and Pedersen \cite{FMP1984} use Go\l\k{a}b's Theorem to show that, given a compact subset $S$ of the plane, there is a shortest connected closed subset of the plane which intersects all lines intersecting $S$.
They give a proof of Go\l\k{a}b's Theorem that depends on another result of Wa\.zewski, namely that a continuum with finite outer linear measure is path-connected.
This is the proof that is presented in Falconer's book \cite[Chapter~3]{Falconer1986}.
Like Go\l\k{a}b's original proof, it uses compactness, but in the form of the Blaschke selection theorem.

Bogn\'ar's proof of his version of Go\l\k{a}b's Theorem \cite{Bognar1989} is somewhat involved and even invokes Zorn's Lemma at some point.
It seems that Bogn\'ar's main purpose was to prove what we have stated as Corollary~\ref{bognartheorem}.
Fremlin \cite{Fremlin1992} gives his own short proof of this corollary and refers to it as \emph{M.~Bogn\'ar's theorem}, although it would be very surprising if it hadn't been known earlier.
Nevertheless, it frequently occurs as a technical result in the literature on existence results in the calculus of variations, for instance, Proposition~2.5 in \cite{DalmasoToader2002} and Lemma~2.6 in \cite{PaoliniStepanov2013}.
According to Go\l\k{a}b \cite{Golab1928}, Wa\.zewski proved the following weaker statement:
\begin{theorem}
Let $C_1\subseteq C_2\subseteq C_3\subseteq $ be an increasing sequence of continua such that $\sup_i L^*(C_i)<\infty$.
Then $L^*(\bigcup_i C_i) = L^*(\overline{\bigcup_i C_i})$.
\end{theorem}
Go\l\k{a}b deduces this statement as an immediate corollary to his theorem.

Ambrosio and Tilli \cite{Ambrosio2004} give a highly analytical proof of Go\l\k{a}b's Theorem, in particular using densities and weak-star convergence of Borel measures.
This proof was corrected by Paolini and Stepanov \cite{PaoliniStepanov2013}.
Alberti and Ottolini \cite{Alberti2017} give a simpler proof which still uses analysis and measure theory, though.

\subsection{Generalizations of Go\l\k{a}b's Theorem}
Ambrosio and Tilli's proof of Go\l\k{a}b's Theorem \cite{Ambrosio2004} is adapted by Paolini and Stepanov \cite{PaoliniStepanov2013} to show the following generalisation of Go\l\k{a}b's Theorem, of which the Euclidean case appears in Dal Maso and Toader \cite{DalmasoToader2002}.
\begin{theorem}
Let $C_n$ be a sequence of closed connected subsets of a complete metric space $X$ and $C$ a closed subset of $X$ such that $\lim_{n\to\infty}H(C_n,C)=0$.
Let $K_n$ be a sequence of closed subsets of $X$ and $K$ a closed subset of $X$ such that $\lim_{n\to\infty}H(K_n,K)=0$.
Then
\[ L^*(C\setminus K) \leq\liminf_n L^*(C_n\setminus K_n).\]
\end{theorem}
A further generalisation is found in Giacomini \cite{Giacomini2002}:
\begin{theorem}
Let $m\in\bN$, $\Omega$ an open and bounded subset of $\bR^2$, and let $\fhi\colon\overline{\Omega}\times\bR^2\to[0,\infty)$ be continuous and such that for each $x\in\overline{\Omega}$, $\fhi(x,\cdot)$ is a norm on $\bR^2$.
Then the functional $K\mapsto\int_K\fhi(x,\nu_x)\,\mathrm{d}H^1(x)$ is lower semicontinuous on the collection of compact subsets of $\overline{\Omega}$ of finite $1$-dimensional Hausdorff measure with at most $m$ connected components.
\end{theorem}
Here $\nu_x$ denotes the unit normal vector of $K$ at $x$ (which exists $H^1$-a.e.\ if $H^1(K)<\infty$).
Giacomini considers $K$ to be a fracture and the above functional is its surface energy.
However, this functional can be interpreted as the ordinary $1$-dimensional Hausdorff measure in the Finsler space on $\overline{\Omega}$ determined by $\fhi$, and therefore, this theorem follows from Go\l{\k{a}}b's Theorem for sets with at most $m$ connected components \cite[Corollary~3.3]{DalmasoToader2002}.

It is easily observed that some sort of connectedness is necessary in Go\l\k{a}b's Theorem.
For example, let $A_n=\set{\frac{1}{n},\frac{2}{n},\dots,\frac{n}{n}}$.
Then $A_n$ converges to $A=[0,1]$, but $L^*(A_n)=0$ and $L^*(A)=1$.

Go\l\k{a}b already notes that there is no straightforward generalisation of his theorem to area.
However, Vitushkin \cite{Vituskin1966} generalises Go\l\k{a}b's Theorem to $k$-dimensional Hausdorff measure in $d$-dimensional Euclidean spaces, by assuming the additional hypothesis that the expected number of connected components of the intersection of $A_n$ by a flat of codimension $k-1$ (with respect to the Haar probability measure on the Grassmannian) is bounded uniformly over $n$.

Dal Maso, Morel and Solimini \cite{Dalmaso1992} generalise Go\l\k{a}b's Theorem for linear measure by assuming that the sets in the sequences uniformly satisfy a certain concentration property.
This is generalised to $k$-dimensional measures by Morel and Solimini \cite[Chapter~10]{Morel1995} and Lops \cite{Lops2005}.


{\small
\begin{thebibliography}{10}

\bibitem{Alberti2017}
G.~Alberti and M.~Ottolini, \emph{On the structure of continua with
  finite length and {G}o\l\k{a}b's semicontinuity theorem}, Nonlinear Anal.\
  \textbf{153} (2017), 35--55. %\MR{3614660}

\bibitem{Ambrosio2004}
L.~Ambrosio and P.~Tilli, \emph{Topics on analysis in metric spaces},
  Oxford Lecture Series in Mathematics and its Applications, vol.~25, Oxford
  University Press, Oxford, 2004. %\MR{MR2039660 (2004k:28001)}

\bibitem{Besicovitch1938}
A.~S. Besicovitch, \emph{On the fundamental geometrical properties of linearly
  measurable plane sets of points ({II})}, Math.\ Ann.\ \textbf{115} (1938),
  no.~1, 296--329. %\MR{MR1513189}

\bibitem{BlumenthalBook}
L.~M.~Blumenthal, \emph{Theory and applications of distance geometry},
  Oxford, at the Clarendon Press, 1953. %\MR{0054981 (14,1009a)}

\bibitem{Bognar1989}
M.~Bogn{\'a}r, \emph{On the exterior linear measure}, Acta Math.\ Hungar.\
  \textbf{53} (1989), no.~1-2, 159--168.

\bibitem{BGTZ}
M.~Brazil, R.~L. Graham, D.~A. Thomas, and M.~Zachariasen, \emph{On the history
  of the {E}uclidean {S}teiner tree problem}, Arch.\ Hist.\ Exact Sci.\
  \textbf{68} (2014), 327--354.

\bibitem{Caratheodory1914}
C.~Carath\'eodory, \emph{{\"Uber das lineare Ma{\ss} von Punktmengen -- eine
  Verallgemeinerung des L\"angenbegriffs.}}, Nachrichten von der Gesellschaft
  der Wissenschaften zu G{\"o}ttingen, Mathematisch-Physikalische Klasse
  \textbf{1914} (1914), 404--426.

\bibitem{Choquet1938}
G.~Choquet, \emph{Etude de certains r\'eseaux de routes}, Comptes Rendus Acad.\
  Sci.\ \textbf{206} (1938), 310--313.

\bibitem{Dalmaso1992}
G.~Dal~Maso, J.-M. Morel, and S.~Solimini, \emph{A variational method in image
  segmentation: existence and approximation results}, Acta Math.\ \textbf{168}
  (1992), no.~1-2, 89--151. %\MR{1149865 (92m:49020)}

\bibitem{DalmasoToader2002}
G.~Dal~Maso and R.~Toader, \emph{A model for the quasi-static growth of
  brittle fractures: existence and approximation results}, Arch.\ Ration.\ Mech.\
  Anal.\ \textbf{162} (2002), no.~2, 101--135. %\MR{1897378 (2003g:74076)}

\bibitem{Ducret2008}
S.~Ducret and M.~Troyanov, \emph{Steiner's invariants and minimal
  connections}, Port.\ Math.\ \textbf{65} (2008), no.~2, 237--242. %\MR{2428417 (2009f:51014)}

\bibitem{Eilenberg1943}
S.~Eilenberg and O.~G. Harrold, Jr., \emph{Continua of finite linear
  measure. {I}}, Amer.\ J. Math.\ \textbf{65} (1943), 137--146. %\MR{MR0007643 (4,172e)}

\bibitem{Ettling1978}
W.~A. Ettling, \emph{On arc length sharpenings}, Pacific J. Math.\
  \textbf{76} (1978), no.~2, 361--370. %\MR{MR506139 (80a:53081)}

\bibitem{FMP1984}
V.~Faber, J.~Mycielski, and P.~Pedersen, \emph{On the shortest curve
  which meets all the lines which meet a circle}, Ann.\ Polon.\ Math.\ \textbf{44}
  (1984), no.~3, 249--266. %\MR{817799 (87b:52023)}

\bibitem{Falconer1986}
K.~J. Falconer, \emph{The geometry of fractal sets}, Cambridge Tracts in
  Mathematics, vol.~85, Cambridge University Press, Cambridge, 1986.
  %\MR{MR867284 (88d:28001)}

\bibitem{Fremlin1992}
D.~H. Fremlin, \emph{Spaces of finite length}, Proc.\ London Math.\ Soc.\ (3)
  \textbf{64} (1992), no.~3, 449--486. %\MR{MR1152994 (93c:28006)}

\bibitem{Frink1936}
O.~Frink, \emph{Geodesic continua in abstract metric space}, Amer.\ J. Math.\
  \textbf{58} (1936), 514--520.

\bibitem{Giacomini2002}
A.~Giacomini, \emph{A generalization of {G}o{\l}{\k{a}}b's theorem and
  application to fracture mechanics}, Math.\ Models Methods Appl.\ Sci.\
  \textbf{12} (2002), 1245--1267.

\bibitem{Gilbert-Pollak1968}
E.~N. Gilbert and H.~O. Pollak, \emph{Steiner minimal trees}, SIAM J. Appl.\
  Math.\ \textbf{16} (1968), 1--29.

\bibitem{Golab1928}
S.~Go{\l}{\k{a}}b, \emph{Sur quelques points de la th{\'e}orie de la longueur},
  Ann.\ Soc.\ Polon.\ Math.\ \textbf{7} (1928), 227--241.

\bibitem{Hausdorff1918}
F.~Hausdorff, \emph{Dimension und \"au\ss eres {M}a\ss}, Math.\ Ann.\
  \textbf{79} (1918), no.~1-2, 157--179. %\MR{1511917}

\bibitem{Ivanov2005}
A.~O. Ivanov, I.~M. Nikonov, and A.~A. Tuzhilin, \emph{Sets admitting
  connection by graphs of finite length}, Mat.\ Sb.\ \textbf{196} (2005), no.~6,
  71--110. %\MR{2164552 (2006g:52015)}

\bibitem{Jones1990}
P.~W. Jones, \emph{Rectifiable sets and the traveling salesman problem},
  Invent.\ Math.\ \textbf{102} (1990), no.~1, 1--15. %\MR{1069238 (91i:26016)}

\bibitem{Klein1984}
E.~Klein and A.~C. Thompson, \emph{Theory of correspondences}, Canadian
  Mathematical Society Series of Monographs and Advanced Texts, John Wiley \&
  Sons Inc., New York, 1984, Including applications to mathematical economics,
  A Wiley-Interscience Publication. %\MR{752692 (86a:90012)}

\bibitem{Koch1904}
Helge~von Koch, \emph{Sur une courbe continue sans tangente, obtenue par une
  construction g\'eom\'etrique \'el\'ementaire}, Ark.\ Mat.\ Astron.\ Fys.\
  \textbf{1} (1904), 681--702.

\bibitem{Kolmogorov1933}
A.~Kolmogoroff, \emph{Beitr\"age zur {M}a{\ss}theorie}, Math.\ Ann.\ \textbf{107}
  (1933), 351--366.

\bibitem{Levi1993}
S.~Levi, R.~Lucchetti, and J.~Pelant, \emph{On the infimum of the {H}ausdorff
  and {V}ietoris topologies}, Proc.\ Amer.\ Math.\ Soc.\ \textbf{118} (1993),
  no.~3, 971--978. %\MR{1165059 (93m:54022)}

\bibitem{Lops2005}
F.~A. Lops, \emph{A denoised version of some kinds of set convergence},
  Adv.\ Nonlinear Stud.\ \textbf{5} (2005), no.~3, 303--335. %\MR{2151759 (2006e:49025)}

\bibitem{Menger1932}
K.~Menger, \emph{Eine neue {D}efinition der {B}ogenl\"ange}, Ergebnisse eines
  Mathematischen Kolloquiums (K.~Menger, ed.), vol.~2, Teubner, Leipzig und
  Wien, 1932, pp.~11--12.

\bibitem{Menger1930}
K.~Menger, \emph{Untersuchungen {\"u}ber allgemeine {M}etrik. {V}ierte
  {U}ntersuchung. {Z}ur {M}etrik der {K}urven}, Math.\ Ann.\ \textbf{103} (1930),
  no.~1, 466--501. %\MR{MR1512632}

\bibitem{Menger1931}
K.~Menger, \emph{Some applications of point-set methods}, Ann.\ of Math.\ (2)
  \textbf{32} (1931), no.~4, 739--760. %\MR{MR1503027}

\bibitem{Michael1951}
E.~Michael, \emph{Topologies on spaces of subsets}, Trans.\ Amer.\ Math.\ Soc.\
  \textbf{71} (1951), 152--182. %\MR{0042109 (13,54f)}

\bibitem{Milgram1940}
A.~N. Milgram, \emph{On shortest paths through a set}, Reports of a
  mathematical colloquium (K.~Menger, ed.), second series, University of Notre
  Dame, 1940. pp.~39--44.

\bibitem{Mimura1933}
Y.~Mimura, \emph{{\"U}ber die {B}ogenl{\"a}nge}, Ergebnisse eines
  Mathematischen Kolloquiums (K.~Menger, ed.), vol.~4, Teubner, Leipzig und
  Wien, 1933. pp.~20--22.

\bibitem{Morel1995}
J.-M.~Morel and S.~Solimini, \emph{Variational methods in image
  segmentation}, Progress in Nonlinear Differential Equations and their
  Applications, 14, Birkh\"auser Boston, Inc., Boston, MA, 1995. %, With seven
  %image processing experiments. %\MR{1321598}

\bibitem{PaoliniUlivi2010}
E.~Paolini and L.~Ulivi, \emph{The {S}teiner problem for infinitely many
  points}, Rend.\ Semin.\ Mat.\ Univ.\ Padova \textbf{124} (2010), 43--56.
  %\MR{2752675 (2012d:52010)}

\bibitem{PaoliniStepanov2013}
E.~Paolini and E.~Stepanov, \emph{Existence and regularity results
  for the {S}teiner problem}, Calc.\ Var.\ Partial Differential Equations
  \textbf{46} (2013), no.~3-4, 837--860. %\MR{3018174}

\bibitem{Schrijver2005}
A.~Schrijver, \emph{On the history of combinatorial optimization (till 1960)},
  Handbooks in Operations Research and Management Science (K.~Aardal, G.~L.
  Nemhauser, and R.~Weismantel, eds.), vol.~12, Elsevier Science Publishers
  B.V., 2005, pp.~1--68.

\bibitem{Schul2007}
R.~Schul, \emph{Analyst's traveling salesman theorems. {A} survey}, In the
  tradition of {A}hlfors-{B}ers. {IV}, Contemp.\ Math., vol.~432, Amer.\ Math.\
  Soc., Providence, RI, 2007, pp.~209--220. %\MR{2342818 (2009b:49099)}

\bibitem{Vituskin1966}
A.~G. Vitushkin, \emph{Proof of the upper semicontinuity of the variation of a set}, Soviet Math.\ Dokl.\ \textbf{7} (1966), 206--209.
  %\MR{0197665 (33 \#5828)}

\bibitem{Wazewski1927}
T.~Wa{\.z}ewski, \emph{Kontinua prostowalne w zwi\k{a}zku z funkcjami i
  odwzorowaniami absolutnie ci\k{a}g{\l}ymi}, Dodatek do Roczn.\ Pol.\ Tow.\ Mat.\
  (1927), 9--49.

\end{thebibliography}
}
\end{document}